\documentclass[12pt,a4paper]{amsart}
\makeatletter
\renewcommand\normalsize{%
    \@setfontsize\normalsize{11.7}{14pt plus .3pt minus .3pt}%
    \abovedisplayskip 10\p@ \@plus4\p@ \@minus4\p@
    \abovedisplayshortskip 6\p@ \@plus2\p@
    \belowdisplayshortskip 6\p@ \@plus2\p@
    \belowdisplayskip \abovedisplayskip}
\renewcommand\small{%
    \@setfontsize\small{9.5}{12\p@ plus .2\p@ minus .2\p@}%
    \abovedisplayskip 8.5\p@ \@plus4\p@ \@minus1\p@
    \belowdisplayskip \abovedisplayskip
    \abovedisplayshortskip \abovedisplayskip
    \belowdisplayshortskip \abovedisplayskip}
\renewcommand\footnotesize{%
    \@setfontsize\footnotesize{8.5}{9.25\p@ plus .1pt minus .1pt}
    \abovedisplayskip 6\p@ \@plus4\p@ \@minus1\p@
    \belowdisplayskip \abovedisplayskip
    \abovedisplayshortskip \abovedisplayskip
    \belowdisplayshortskip \abovedisplayskip}
\ifdefined\pdfpagewidth
\else
\fi
\calclayout
\makeatother

\usepackage{amsthm,amssymb}
\usepackage[T1]{fontenc}
\usepackage{lmodern}
\usepackage[utf8]{inputenc}     
\usepackage{xcolor}
\usepackage{enumerate}
\usepackage[
  pdftitle={C*-Rigidity of bounded geometry metric spaces},
  pdfauthor={Diego Mart\'{i}nez, Federico Vigolo},
  pdfsubject={Mathematics},
  colorlinks=true, linkcolor=blue, linkbordercolor=blue, citecolor=red, citebordercolor=red, urlcolor=blue, linktocpage=true
]{hyperref}
\usepackage[capitalise, nameinlink, noabbrev, nosort]{cleveref}
\usepackage[lite,initials]{amsrefs}

\usepackage{mathtools}          
\usepackage{microtype}          
\usepackage{bbm}          
\usepackage{bm}           
\usepackage[shortcuts]{extdash}       
\usepackage{accents}
\usepackage{xspace}

\usepackage{tikz-cd}

\usepackage[inline]{enumitem}
\setlist[enumerate,1]{font=\normalfont}
\crefname{enumi}{}{} 
\crefname{enumi}{}{} 

\newcommand{\C}{\mathbb{C}}

\makeatletter
\def\appmap{\@ifnextchar[{\@withappmap}{\@withoutappmap}}
  \def\@withappmap[#1]#2#3#4{{{f}^{#1}_{#2, #3, #4}}}
  \def\@withoutappmap#1#2#3{{{f}_{#1, #2, #3}}}
\def\cappmap{\@ifnextchar[{\@withcappmap}{\@withoutcappmap}}
  \def\@withcappmap[#1]#2#3#4{{\crse{f}^{#1}_{#2, #3, #4}}}
  \def\@withoutcappmap#1#2#3{{\crse{f}_{#1, #2, #3}}}
\def\crseimg{\@ifnextchar[{\@withcrseimg}{\@withoutcrseimg}}
  \def\@withcrseimg[#1]#2#3{{I^{#1}_{#2, #3}}}
  \def\@withoutcrseimg#1#2{{I_{#1, #2}}}
\makeatother

\newcommand*{\cstar}{\texorpdfstring{$C^*$\nobreakdash-\hspace{0pt}}{*-}}

\newcommand*{\uroecstar}[1]{C^*_{\rm u}{\left(#1\right)}}

\makeatletter
  \def\roeclike{\@ifnextchar[{\@withroeclike}{\@withoutroeclike}}
    \def\@withroeclike[#1]#2{\CR^*_{#1}{\left(#2\right)}}
    \def\@withoutroeclike#1{\CR^*{\left(#1\right)}}
  \def\roeclikeone{\@ifnextchar[{\@withroeclikeone}{\@withoutroeclikeone}}
    \def\@withroeclikeone[#1]#2{\CR^*_{#1}{\left(#2\right)}}
    \def\@withoutroeclikeone#1{\CR^*_1{\left(#1\right)}}
  \def\roecliketwo{\@ifnextchar[{\@withroecliketwo}{\@withoutroecliketwo}}
    \def\@withroecliketwo[#1]#2{\CR^*_{#1}{\left(#2\right)}}
    \def\@withoutroecliketwo#1{\CR^*_2{\left(#1\right)}}
  \def\roecstar{\@ifnextchar[{\@withroecstar}{\@withoutroecstar}}
    \def\@withroecstar[#1]#2{C^*_{#1,\rm Roe}{\left(#2\right)}}
    \def\@withoutroecstar#1{C^*_{\rm Roe}{\left(#1\right)}}
  \def\roestar{\@ifnextchar[{\@withroestar}{\@withoutroestar}}
    \def\@withroestar[#1]#2{\C_{#1,\rm Roe}{\left[#2\right]}}
    \def\@withoutroestar#1{\C_{\rm Roe}{\left[#1\right]}}
  \def\classicalroecstar{\@ifnextchar[{\@withclassicalroecstar}{\@withoutclassicalroecstar}}
    \def\@withclassicalroecstar[#1]#2{C^*{\left(#2\right)}}
    \def\@withoutclassicalroecstar#1{C^*{\left(#1\right)}}
  \def\cpcstar{\@ifnextchar[{\@withcpcstar}{\@withoutcpcstar}}
    \def\@withcpcstar[#1]#2{C^*_{#1,\rm cp}{\left(#2\right)}}
    \def\@withoutcpcstar#1{C^*_{\rm cp}{\left(#1\right)}}
  \def\cpstar{\@ifnextchar[{\@withcpstar}{\@withoutcpstar}}
    \def\@withcpstar[#1]#2{\C_{#1,\rm cp}{\left[#2\right]}}
    \def\@withoutcpstar#1{\C_{\rm cp}{\left[#1\right]}}
  \def\qlcstar{\@ifnextchar[{\@withqlcstar}{\@withoutqlcstar}}
    \def\@withqlcstar[#1]#2{C^*_{#1,\rm ql}{\left(#2\right)}}
    \def\@withoutqlcstar#1{C^*_{\rm ql}{\left(#1\right)}}
\makeatother

\newcommand*{\coe}[1]{{\rm CE}{\left(#1\right)}}

\newcommand*{\charfunc}[1]{\mathbbm{1}_{#1}}                        \newcommand*{\chf}[1]{\charfunc{#1}}
\newcommand*{\charfunccomp}[2]{\charfunc{#1 \smallsetminus #2}}     
\newcommand*{\charfunccompX}[1]{\charfunccomp{X}{#1}}               \newcommand*{\chfcX}[1]{\charfunccompX{#1}}

\newcommand*{\matrixunit}[2]{e_{#1,#2}}

\definecolor{newpurple}{rgb}{0.8, 0, 0.9}

        \newcommand{\NN}{\mathbb{N}}

       \newcommand{\CB}{\mathcal{B}}

       \newcommand{\CH}{\mathcal{H}}
       
\newcommand{\CK}{\mathcal{K}}       
\newcommand{\CM}{\mathcal{M}}       
       
       \newcommand{\CR}{\mathcal{R}}

\DeclarePairedDelimiter\abs{\lvert}{\rvert}     
\DeclarePairedDelimiter\norm{\lVert}{\rVert}        
\DeclarePairedDelimiter\angles{\langle}{\rangle}    

\DeclarePairedDelimiter\paren{(}{)}         
\DeclarePairedDelimiter\braces{\{}{\}}          

   \newcommand{\bigparen}[1]{\paren[\big]{#1}}

\newcommand{\scal}[2]{\angles{#1,#2}}           

\DeclareMathOperator{\id}{id}               

\DeclareMathOperator{\aut}{Aut}             
\DeclareMathOperator{\out}{Out}             

\DeclareMathOperator{\diam}{diam}           

\theoremstyle{plain}
\newtheorem{thm}{Theorem}[section]      \newtheorem{theorem}[thm]{Theorem}
     \newtheorem{proposition}[thm]{Proposition}
            \newtheorem{lemma}[thm]{Lemma}

\newtheorem{claim}[thm]{Claim}

\newtheorem*{thm*}{Theorem}         \newtheorem*{theorem*}{Theorem}
\newtheorem*{prop*}{Proposition}        \newtheorem*{proposition*}{Proposition}
\newtheorem*{lem*}{Lemma}           \newtheorem*{lemma*}{Lemma}
\newtheorem*{cor*}{Corollary}           \newtheorem*{corollary*}{Corollary}
\newtheorem*{qu*}{Question}         \newtheorem*{question*}{Question}
\newtheorem*{conj*}{Conjecture}         \newtheorem*{conjecture*}{Question}
\newtheorem*{fact*}{Fact}
\newtheorem*{claim*}{Claim}
\newtheorem*{case*}{Case}
\newtheorem*{problem*}{Problem}

\numberwithin{equation}{section}

\newtheorem{alphthm}{Theorem}           

\theoremstyle{definition}
        \newtheorem{definition}[thm]{Definition}

      \newtheorem{convention}[thm]{Convention}

\newtheorem*{de*}{Definition}           \newtheorem{definition*}{Definition}
\newtheorem*{notation*}{Notation}
\newtheorem*{conv*}{Convention}         \newtheorem*{convention*}{Convention}

\theoremstyle{remark}
           \newtheorem{remark}[thm]{Remark}

\newtheorem*{remark*}{Remark}

\DeclareMathAlphabet{\mathbit}{OT1}{cmr}{bx}{it}

\crefname{thm}{Theorem}{Theorems}              \crefname{theorem}{Theorem}{Theorems}
\crefname{prop}{Proposition}{Propositions}     \crefname{proposition}{Proposition}{Propositions}
\crefname{lem}{Lemma}{Lemmas}                  \crefname{lemma}{Lemma}{Lemmas}
\crefname{rmk}{Remark}{Remarks}                \crefname{remark}{Remark}{Remarks}
\crefname{cor}{Corollary}{Corollaries}         \crefname{corollary}{Corollary}{Corollaries}
\crefname{qu}{Question}{Questions}             \crefname{question}{Question}{Questions}
\crefname{conj}{Conjecture}{Conjectures}       \crefname{conjecture}{Conjecture}{Conjectures}
\crefname{fact}{Fact}{Facts}
\crefname{claim}{Claim}{Claims}
\crefname{case}{Case}{Cases}
\crefname{alphthm}{Theorem}{Theorems}          \crefname{alphcor}{Corollary}{Corollaries}
\crefname{alphprop}{Proposition}{Propositions}

\newcommand{\Ad}{{\rm Ad}}

\newcommand{\mhyphen}{\textnormal{-}}
\newcommand{\variable}{\,\mhyphen\,}

\newcommand*{\munit}[2]{\matrixunit{#1}{#2}}

\newcommand*{\ops}[1]{\CB(\ell^2(#1;\CH))}
\newcommand*{\elltwo}[1]{\ell^2(#1;\CH)}
\newcommand*{\Prop}[1]{{\rm Prop}(#1)}

\begin{document}

\title[C*-rigidity of bounded geometry metric spaces]{C*-rigidity of bounded geometry metric spaces}
\date{\today}

\author[Diego Mart\'{i}nez]{Diego Mart\'{i}nez $^{1}$}
\address{Department of Mathematics, KU Leuven, Celestijnenlaan 200B, 3001 Leuven, Belgium.}
\email{diego.martinez@kuleuven.be}

\author[Federico Vigolo]{Federico Vigolo $^{2}$}
\address{Mathematisches Institut, Georg-August-Universit\"{a}t G\"{o}ttingen, Bunsenstr. 3-5, 37073 G\"{o}ttingen, Germany.}
\email{federico.vigolo@uni-goettingen.de}

\begin{abstract}
  We prove that uniformly locally finite metric spaces with isomorphic Roe algebras must be coarsely equivalent. As an application, we also prove that the outer automorphism group of the Roe algebra of a metric space of bounded geometry is canonically isomorphic to the group of coarse equivalences of the space up to closeness.
\end{abstract}

%
%


\thanks{{$^{1}$} Funded by the Deutsche Forschungsgemeinschaft (DFG, German Research Foundation) under Germany’s Excellence Strategy – EXC 2044 – 390685587, Mathematics Münster – Dynamics – Geometry – Structure; the Deutsche Forschungsgemeinschaft (DFG, German Research Foundation) – Project-ID 427320536 – SFB 1442, and ERC Advanced Grant 834267 - AMAREC}

\thanks{{$^{2}$}  Funded by the Deutsche Forschungsgemeinschaft (DFG) -- GRK 2491:  Fourier Analysis and Spectral Theory -- Project-ID 398436923}

\maketitle

\section{Introduction}

The problem of \cstar rigidity lies at the interface between two seemingly unrelated worlds. Namely, that of \cstar algebras and that of coarse geometry.
Starting from the latter, \emph{coarse geometry} is the paradigm of studying (metric) spaces by ignoring their ``local'' properties and only investigating their ``large-scale'' geometric features.

More formally, a map \(f \colon X \to Y\) between two metric spaces is \emph{controlled} if for every \(r \geq 0\) there is some \(R \geq 0\) such that for every pair \(x_1, x_2 \in X\) with \(d(x_1, x_2) \leq r\), the images satisfy \(d(f(x_1), f(x_2)) \leq R\).
Two functions $f_1,f_2\colon X \to Y$ are \emph{close} if \(\sup_{x \in X} d(f_1(x), f_2(x)) < \infty\), and two metric spaces $X$ and $Y$ are \emph{coarsely equivalent} if there exist controlled maps \(f \colon X \to Y\) and \(g \colon Y \to X\) such that \(f \circ g\) is close to \(\id_Y\) and \(g \circ f\) is close to \(\id_X\).
The \emph{coarse geometric} properties of a metric space are those properties that are preserved under coarse equivalence.
A prototypical example of coarse equivalence is given by the inclusion $\mathbb  Z\hookrightarrow \mathbb R$ or, more generally, well-behaved discretizations of continuous spaces.

At first sight, coarse equivalence is an extremely weak notion. However, if the space is equipped with additional structure, such as a group action, it is often possible to extract an impressive amount of information from its large scale geometry.
In fact, geometric group theory shows that the coarse geometric setup provides the correct framework to conflate between groups and spaces. The power of this point of view and the breadth of its applications can be easily inferred from any of the numerous books on the subject \cites{delaHarpe2000topics,dructu2018geometric,gromov-1993-invariants,roe_lectures_2003,bridson2013metric}.

On the operator-algebraic side, the main character is the \emph{Roe algebra} $\roecstar X$. Its origin comes from differential geometry, and can be traced back to \cites{RoeIndexI,roe-1993-coarse-cohom}, where Roe used the $K$-theory of such \cstar{}algebras as receptacles for higher indices of differential operators on Riemannian manifolds.
It was then shown that the $K$-theory of $\roecstar X$ can be related with a coarse $K$-homology of $X$ via a certain assembly map. The study of this map is a subject of prime importance, as it can be used, for instance, as a tool to uncover deep interplays between topological and analytical properties of manifolds and prove the Novikov Conjecture \cites{aparicio_baum-connes_2019,higson-roe-1995-coarse-bc,yu1995coarse,yu_coarse_2000,skandalis-tu-yu-coarse-gpds-02,yu-1998-novikov-groups,roe-1993-coarse-cohom,yu_localization_1997}. More recently, Roe algebras have also been proposed to model topological phases in mathematical physics \cite{ewert2019coarse}.

Besides the Roe algebra $\roecstar X$, other Roe-like algebras such as the \emph{uniform Roe algebra} $\uroecstar{X}$ and the \cstar{}algebra of operators of controlled propagation $\cpcstar{X}$ (also known as ``band-dominated operators'') also found a solid place in the mathematical landscape, and have been recognized to be algebras of operators worthy of being studied in their own right \cites{roe_lectures_2003,braga_rigid_unif_roe_2022,li-khukhro-vigolo-zhang-2021,guentner-et-al-2012-geometric-complex,willett-2009-some-notes-prop-a,sako-2014-proper-a-oper-norm}.
We refer to \cref*{sec: prelims} for definitions.

The existence of bridges between the operator algebraic and coarse geometric worlds has been known for a very long time.
It was observed very early on that coarsely equivalent metric spaces always have isomorphic Roe algebras \cites{roe-1993-coarse-cohom,higson-roe-yu-1993-mayer-vietoris}. The \emph{\cstar{}rigidity problem} asks whether the converse is also true.
This fundamental problem and its counterparts dealing with other Roe-like algebras have been studied extensively. After the pioneering work \cite{spakula_rigidity_2013}, a sequence of papers gradually improved the state of the art by proving \cstar{}rigidity in more and more general settings \cites{braga2021uniform,braga2020embeddings,braga_farah_vignati_2022,braga_gelfand_duality_2022,braga_farah_rig_2021,spakula_maximal_2013,braga2020coarse,li-spakula-zhang-2023-measured-asym-exp,jiang2023rigidity}, with a final breakthrough obtained in \cite{braga_rigid_unif_roe_2022}, where the \cstar{}rigidity problem is solved for uniform Roe algebras of uniformly locally finite spaces. Related work also includes \cites{white_cartan_2018,bbfvw_2023_embeddings_vna,chung2018rigidity,baudier2023coarse,braga2024operator}.

The main contribution of this work is the complete solution of the \cstar{}rigidity problem for bounded geometry metric spaces.

\begin{alphthm}\label{thm:intro:rigidity}
    Let \(X\) and \(Y\) be uniformly locally finite metric spaces.
    If $\roecstar X\cong\roecstar Y$, then $X$ and $Y$ are coarsely equivalent. Moreover, the same holds if $\roecstar\variable$ is replaced with $\uroecstar\variable$ or $\cpcstar\variable$.
\end{alphthm}

\begin{remark}
    The Roe algebra $\roecstar X$ can be defined for any proper metric space $X$. However, the most important spaces in view of applications are those of \emph{bounded geometry} (\emph{e.g.}\ covers of compact Riemannian manifolds). A space has bounded geometry if and only if it is coarsely equivalent to a uniformly locally finite metric space. Since coarsely equivalent metric spaces have isomorphic Roe algebras, \cref{thm:intro:rigidity} does indeed solve the \cstar{}rigidity problem for bounded geometry metric spaces.

    As a matter of fact, the techniques here introduced can be adapted to prove \cstar{}rigidity for arbitrary proper metric spaces. However, doing so requires overhauling an important amount of existing literature, and cannot be done in a short space. To keep this paper brief and clear, we decided to leave such an endeavour for a different work \cite{rigid}.  
\end{remark}

The case $\uroecstar\variable$ of \cref{thm:intro:rigidity} is the main theorem of \cite{braga_rigid_unif_roe_2022}.
To a large extent, the strategy of proof follows the same route that, starting with \cite{spakula_rigidity_2013}, all the works on \cstar{}rigidity took. Our main technical contribution is the proof of an unconditional ``concentration inequality'' (cf.\ \cref{prop:conc ineq}) which represents the last piece of the puzzle in the construction of coarse equivalences.

\

Going beyond the problem of \cstar{}rigidity, one disappointing aspect of this picture is a severe lack of functoriality. For instance, while it is true that with every coarse equivalence $f\colon X\to Y$ one can associate an isomorphism $\roecstar X\to\roecstar Y$, this choice is highly non-canonical and very poorly behaved with respect to composition.
On the other hand, it was observed in \cite{braga_gelfand_duality_2022} that this ambiguity vanishes up to innerness. Namely, there is a natural group homomorphism $\tau \colon \coe X\to \out(\roecstar X)$ from the group of closeness-classes of coarse equivalences to the group of outer automorphisms, which is the quotient $\aut(\roecstar X)/\CM(\roecstar X)$ of the group of automorphisms of $\roecstar X$ modulo inner automorphisms of its multiplier algebra (innerness is taken in the multiplier algebra, as $\roecstar X$ is not unital).
It is also proved in \cite{braga_gelfand_duality_2022}*{Theorem B} that the map $\tau$ is in fact an isomorphism for uniformly locally finite metric spaces with \emph{property~A} (see, \emph{e.g.}\ \cites{yu_coarse_2000,willett-2009-some-notes-prop-a,roe_ghostbusting_2014}).
The second contribution of the present work shows that this result holds in complete generality as well.

\begin{alphthm}\label{thm:intro:iso}
    If $X$ is a uniformly locally finite metric space, there is a canonical isomorphism
    \[
    \tau \colon \coe X\xrightarrow{\ \cong\ } \out(\roecstar X).
    \]
\end{alphthm}

\cref{thm:intro:iso} is obtained by proving a refinement of \cref{thm:intro:rigidity} which we find of independent interest (cf.\ \cref{thm: refined rigidity}). This result applies to $\out(\cpcstar X)$ as well.

\subsection*{Acknowledgements}
We would like to thank an anonymous referee for their proof and insights on \cref{lemma:cotype 2}.

\section{Preliminaries}\label{sec: prelims}
This section briefly covers the necessary background for the paper. We refer the reader to \cites{braga_farah_rig_2021,braga_gelfand_duality_2022,spakula_maximal_2013,roe_lectures_2003,roe-algs,willett_higher_2020} (and references therein) for a longer discussion on these topics.
Throughout, $X$ and $Y$ denote metric spaces and their metrics will be $d_X$ and $d_Y$ respectively.

\begin{definition}
    A metric space \(X\) is \emph{uniformly locally finite} if \(\sup_{x \in X} \abs*{\overline{B}(x; R)} < \infty\) for all \(R \geq 0\), where 
    \[
    \overline{B}(x; R)\coloneqq \braces{x'\in X\mid d_X(x,x')\leq R}
    \]
    denotes the closed \(R\)-ball around \(x\) and \(\abs{A}\) is the cardinality of \(A \subseteq X\).
\end{definition}

As mentioned in the introduction, every bounded geometry metric space is coarsely equivalent to a uniformly locally finite metric space. Since this is the only case we are focusing on, we shall work under the following.

\begin{convention}
    $X$ and $Y$ denote uniformly locally finite metric spaces (in particular, they are countable).
\end{convention}

For simplicity, we will only define Roe algebras in the setting above. A more general treatment can be found \emph{e.g.}\ in \cites{willett_higher_2020,roe-algs}.

\begin{remark}
    In the following, uniform local finiteness is only needed in \cref{thm: uniformization}, the rest of the arguments work equally well for all locally finite metric spaces.
\end{remark}

We let \(\CH\) denote an arbitrary (but fixed) Hilbert space. For any \(A \subseteq X\), \(\chf{A} \in \CB(\ell^2(X; \CH))\) is the orthogonal projection onto \(\ell^2(A; \CH) \subseteq \ell^2(X; \CH)\). For ease of notation, we also let \(\chf{x} \coloneqq \chf{\{x\}}\) for all \(x \in X\).

\begin{definition} \label{def:loc comp}
    We say \(t \in \CB(\elltwo X)\) is \emph{locally compact} if \(t \chf{x}\) and \(\chf{x} t\) are compact operators for every \(x \in X\).
\end{definition}

\begin{definition} \label{def:approx ops}
    Let \(t \in \CB(\elltwo X), R \geq 0\) and \(\varepsilon > 0\).
    \begin{enumerate}[label=(\roman*)]
        \item \(t\) has \emph{propagation at most \(R\)} (denoted $\Prop t\leq R$) if \(\chf{B} t \chf{A} = 0\) for every \(A, B \subseteq X\) such that \(d_X(A, B) > R\).
        \item \(t\) has \emph{controlled propagation} if it has propagation at most \(R\) for some \(R \geq 0\).
        \item \(t\) is \emph{\(\varepsilon\)-\(R\)-approximable} if there is some \(s \in \CB(\elltwo X)\) of propagation at most \(R\) such that \(\norm{s - t} \leq \varepsilon\).
        \item \(t\) is \emph{approximable} if for all \(\varepsilon > 0\) there is some \(R \geq 0\) such that \(t\) is \(\varepsilon\)-\(R\)-approximable.
    \end{enumerate}
\end{definition}

\begin{remark}
    Note that if \(\CH\) is infinite dimensional then the identity operator is \emph{not} locally compact, but it does have propagation \(0\).
\end{remark}

We may now define ``Roe-like'' algebras depending on \(\CH\).
\begin{definition} \label{def:roe like algebras}
    Let \(X\) be a uniformly locally finite metric space.
    \begin{enumerate}[label=(\roman*)]
        \item \(\roecstar{X; \CH}\) is the \cstar subalgebra of \(\CB(\elltwo X)\) generated by the locally compact operators of controlled propagation.
        \item \(\cpcstar{X; \CH}\) is the \cstar subalgebra of \(\CB(\elltwo X)\) generated by the operators of controlled propagation.
    \end{enumerate}
\end{definition}

The Roe-like algebras discussed in the introduction are defined making the following choices of coefficients:
\begin{enumerate}[label=(\roman*)]
  \item \(\roecstar{X} \coloneqq \roecstar{X; \ell^2(\NN)};\)
  \item \(\cpcstar{X} \coloneqq \cpcstar{X; \ell^2(\NN)}; \)
  \item \(\uroecstar{X} \coloneqq \roecstar{X; \C} = \cpcstar{X; \C}.\)
\end{enumerate}
For the sake of clarity and generality, in the rest of this paper we will keep the dependence on $\CH$ explicit.

\begin{remark}
    We briefly observe the following.
    \begin{enumerate}[label=(\roman*)]
        \item \(\CH\) is finite dimensional if and only if \(\roecstar{X; \CH} = \cpcstar{X; \CH}\).
        \item \(\cpcstar{X; \CH}\) can also be defined as the set of approximable operators (cf.\ \cref{def:approx ops}).
        \item It is routine to check
        that every compact operator is in \(\roecstar{X; \CH}\). Likewise, it is also clear that \(\ell^\infty(X, \CK(\CH)) \subseteq \roecstar{X; \CH} \subseteq \cpcstar{X; \CH}\).
    \end{enumerate}
\end{remark}

The following weakening of \(\varepsilon\)-\(R\)-approximability will be of use in \cref{prop:conc ineq}.
\begin{definition} \label{def: quasi local}
    Let \(t \in \CB(\elltwo X), R \geq 0\) and \(\varepsilon > 0\). We say \(t\) is \emph{\(\varepsilon\)-\(R\)-quasi-local} if \(\norm{\chf{B} t \chf{A}} \leq \varepsilon\) for all \(A, B \subseteq X\) such that \(d_X(A, B) > R\).
\end{definition}

Observe that every \(\varepsilon\)-\(R\)-approximable operator is \(\varepsilon\)-\(R\)-quasi-local as well.

\begin{remark}
    Analogously to $\cpcstar{X,\CH}$, one can also consider the \cstar algebra of all ``quasi-local operators'', and show that the \cstar rigidity phenomenon applies in that case as well. A unified approach to proving \cstar rigidity simultaneously for all these Roe-like algebras is the subject of \cite{rigid}.
\end{remark}

The following is one of the key notions when discussing rigidity questions.
\begin{definition}\label{def: weak appr ops}
    A bounded operator $T\colon \elltwo X\to \elltwo Y$ is \emph{weakly approximately controlled} if for every $r\geq 0$ and $\varepsilon>0$ there is some $R\geq 0$ such that $\Ad(T)$ maps contractions of $r$-controlled propagation to $\varepsilon$-$R$-approximable operators:
    \[
    \{t\in\ops X\mid \norm{t}\leq 1,\ \Prop{t}\leq r \} \xrightarrow{\Ad(T)}
    \{\text{$\varepsilon$-$R$-approximable operators}\}.
    \]
\end{definition}

\begin{remark}
    In the terminology of \cite{braga_gelfand_duality_2022}*{Definition 3.1}, $T$ is weakly approximately controlled if and only if $\Ad(T)$ is \emph{coarse-like}.
\end{remark}

For the purposes of this text, the main interest of weakly controlled operators is the following.



\begin{lemma}\label{lem:approximationa are controlled}
    Let $T\colon \elltwo{X}\to \elltwo{Y}$ be weakly approximately controlled. Then for every $r\geq 0$ and $\delta>0$ there is an $R\geq 0$ such that if $A\subseteq X$ has $\diam(A) \leq r$ and $C,C'\subseteq Y$ are such that
    \[
    \norm{\chf C T\chf A}, \norm{\chf{C'}T\chf A} \geq \delta,
    \]
    then $d_Y(C,C')\leq R$.
\end{lemma}
\begin{proof}
    Let $v,v'\in\elltwo{A}$ be two unit vectors such that $\norm{\chf C T(v)},\norm{\chf{C'} T(v')}\geq \delta/2$. Let $w=T(v)$ and $w'=T(v')$ be their images.
    Consider the rank-1 contractions
    \[
        \munit v{v'}(\variable) \coloneqq \scal{v'}{\variable} v
        \; \text{ and } \;
        \munit w{w'}(\variable) \coloneqq \scal{w'}{\variable} w.
    \]
    Observe that $\Ad(T)$ maps $\munit{v}{v'}$ to $\munit{w}{w'}$, and that
    \[
    \norm{\chf C \munit{w}{w'}\chf{C'}}= \norm{\chf C(w)}\norm{\chf{C'}(w')}\geq \frac{\delta^2}4.
    \]
    This shows that $\Ad(T)(\munit{v}{v'})$ is not $\frac{\delta^2}4$-$R$-quasilocal for any $R< d_Y(C,C')$.
    Since $\munit{v}{v'}$ is a contraction of propagation bounded by $\diam(A)\leq r$, the weak approximability condition on $T$ yields the desired uniform upper bound on $d_Y(C,C')$. 
\end{proof}

We will make use of the following results.

\begin{theorem}[cf.\ \cite{spakula_rigidity_2013}*{Lemma 3.1} and \cite{braga2020embeddings}*{Lemma 6.1}]\label{thm: implementation}
    Any isomorphism $\Phi\colon \roecstar{X; \CH} \to \roecstar{Y; \CH}$ is spatially implemented. That is, there exists a unitary operator $U\colon\elltwo{X}\to \elltwo{Y}$ such that $\Phi=\Ad(U)|_{\roecstar{X; \CH}}$. 

    Moreover, the same is true if $\roecstar{\variable; \CH}$ is replaced by $\cpcstar{\variable; \CH}$.
\end{theorem}

\begin{theorem}[cf.\ \cite{braga_gelfand_duality_2022}*{Theorems 3.4 and 3.5}]\label{thm: uniformization}
    If $U\colon \elltwo{X}\to \elltwo{Y}$ is a unitary such that $\Ad(U)$ implements an isomorphism $\roecstar{X; \CH} \cong \roecstar{Y; \CH}$, then $U$ is weakly approximately controlled. 
    
    Moreover, the same is true if $\roecstar{\variable; \CH}$ is replaced by $\cpcstar{\variable; \CH}$.
\end{theorem}

\section{Proof of C*-rigidity}\label{sec: rigidity}
In this section we prove \cref{thm:intro:rigidity}.
The proof will start as usual, namely by applying \cref{thm: implementation,thm: uniformization} to pass from an isomorphism of \cstar algebras to a well-behaved unitary operator.
It then remains to use this operator to construct a coarse equivalence.
In order to do this, we first need to prove \cref{prop:conc ineq}, which is the key new technical step in the proof of \cref{thm:intro:rigidity}.

\subsection{A concentration inequality}
The following lemma leverages the fact that Hilbert spaces have cotype 2.

\begin{lemma} \label{lemma:cotype 2}
    Let \((v_n)_{n \in \NN} \subseteq \CH\) be a sequence of vectors of the Hilbert space \(\CH\) with square-summable norms. Then
    \[
        \sup_{\varepsilon_n = \pm 1} \norm{\sum_{n \in \NN} \varepsilon_n v_n}^2 \geq \sum_{n \in \NN} \norm{v_n}^2,
    \]
    where the supremum is taken among all possible \((\varepsilon_n)_{n \in \NN} \in \{-1, +1\}^\NN\).
\end{lemma}

\begin{proof}
    Let \(\Omega \coloneqq \{-1, +1\}^\NN\) be equipped with the usual product probability measure, and recall the identification $L^2(\Omega)\otimes\CH = L^2(\Omega;\CH)$, where the latter is given the norm $\norm{F}^2\coloneqq\int_\Omega \norm{F(\varepsilon)}^2d\varepsilon$.

    The natural projections \(r_n \colon \Omega \to \{\pm 1\}\) (the Rademacher functions) are orthonormal in $L^2(\Omega)$. It then follows from square-summability that the sum $F\coloneqq \sum_{n\in\NN} r_n\otimes v_n$ gives a well-defined element of $L^2(\Omega)\otimes\CH$ of square-norm 
    \[
    \norm{F}^2 = \sum_{n\in\NN}\norm{r_n\otimes v_n}^2 = \sum_{n\in\NN}\norm{v_n}^2.
    \]
    On the other hand, when seen in $L^2(\Omega;\CH)$, the element $F$ is the function $F(\varepsilon)= \sum_{n \in \NN} \varepsilon_n v_n$. Computing its norm in $L^2(\Omega;\CH)$ then yields
    \[
        \int_{\Omega} \norm{\sum_{n \in \NN} \varepsilon_n v_n}^2 d\varepsilon = \sum_{n \in \NN} \norm{v_n}^2.
    \]
    The lemma now follows, as the supremum is at least as large as the average.
\end{proof}

\begin{proposition}\label{prop:conc ineq}
    Let $U\colon \elltwo{X}\to \elltwo{Y}$ be a unitary. Given $\delta>0$ and $R>0$, suppose there is some $y\in Y$ such that
    \[
    \norm{\chf{\overline{B}(y;R)}U\chf x}\leq \delta \quad \text{for all } x\in X.
    \]
    Then there is some $A \subseteq X$ such that $U\chf{A}U^*$ is not $\frac 12(1-\delta^2)^{1/2}$-$R$-quasilocal.
\end{proposition}
\begin{proof}
    Let \(h \in \CH\) be an arbitrary vector of norm \(1\) and consider $\delta_y\otimes h\in \ell^2(Y,\CH)$.
    Observe that \(v \coloneqq U^*(\delta_y \otimes h) \in \ell^2(X; \CH)\) also has norm \(1\) and, by construction,
    \(\delta_y \otimes h = U(v) = \chf{y} U (v)\).
    Letting \(B \coloneqq \overline{B}(y; R)\) one has that
    \begin{align}\label{eq:conc ineq:1}
        \begin{split}
            \norm{\left(1 - \chf{B}\right) U \chf{x} (v)}^2 
            & = \norm{U \chf{x} (v)}^2 - \norm{\chf{B} U \chf{x} (v)}^2  \\
            & \geq \norm{\chf{x} (v)}^2 - \delta^2 \norm{\chf{x} (v)}^2 
            = \left(1 - \delta^2\right) \norm{\chf{x} (v)}^2, 
        \end{split}
    \end{align}
    where the inequality follows from the hypothesis of the proposition. Applying \cref{lemma:cotype 2} to the family \(((1-\chf{B}) U \chf{x} (v))_{x \in X}\) one obtains that there is some choice of signs \((\varepsilon_x)_{x \in X} \in \{-1, +1\}^X\) such that
    \begin{align}\label{eq:conc ineq:2}
        \begin{split}
        \norm{\sum_{x \in X} \varepsilon_x \left(1 - \chf{B}\right) U \chf{x} (v)}^2 
        & \geq \sum_{x \in X} \norm{\left(1-\chf{B}\right) U \chf{x} (v)}^2    \\
        & \geq \left(1 - \delta^2\right) \sum_{x \in X} \norm{\chf{x} (v)}^2 
        = 1- \delta^2, 
        \end{split}
    \end{align}
    where the second inequality is given by \eqref{eq:conc ineq:1} and the last equality follows since \(\sum_{x \in X} \chf{x} v = v\) has norm \(1\).
    
    Partition \(X = P \sqcup N\), where \(P \coloneqq \{x \in X \mid \varepsilon_x = +1\}\) and \(N \coloneqq X \smallsetminus P\), and observe that
    \[
        \norm{\sum_{x \in X} \varepsilon_x \left(1 - \chf{B}\right) U \chf{x} (v)}
         \leq \norm{\sum_{x \in P} \varepsilon_x \left(1 - \chf{B}\right) U \chf{x} (v)} + \norm{\sum_{x \in N} \varepsilon_x \left(1 - \chf{B}\right) U \chf{x} (v)}.
    \]
    Then \eqref{eq:conc ineq:2} implies that, for either $A=P$ or \(A=N\), we have
    \[
        \norm{\sum_{x \in A} \left(1 - \chf{B}\right) U \chf{x} (v)}
         \geq \frac 12 \left(1-\delta^2\right)^{1/2}.
    \]
    We claim that \(A \subseteq X\) has the desired property, \emph{i.e.}\ some corner of \(U \chf{A} U^*\) has ``large'' norm. We just check this on the sets \(\{y\}\) and \(Y \smallsetminus B\):
    \begin{align*}
        \norm{\chf{Y \smallsetminus B} U \chf{A} U^* \chf{y}} 
        & \geq \norm{\chf{Y \smallsetminus B} U \chf{A} U^* \chf{y} (\delta_y\otimes h)}
        = \norm{\chf{Y \smallsetminus B} U \chf{A} (v)} \\
        &= \norm{\sum_{x \in A} \left(1 - \chf{B}\right) U \chf{x} (v)} 
        \geq \frac 12 \left(1-\delta^2\right)^{1/2}.
    \end{align*}
    Since $X$ is locally finite and \(B = \overline{B}(y; R)\), it follows that \(d_Y(y, Y \smallsetminus B) > R\) and the above computation proves the claim.
\end{proof}

\subsection{Completing the proof}
Now that \cref{prop:conc ineq} has been shown, completing the proof of \cstar rigidity is a standard routine which essentially relies on the arguments in \cites{spakula_rigidity_2013,braga_farah_rig_2021}.
However, the concluding part of \cite{spakula_rigidity_2013}*{Theorem 4.1 and Lemma 4.5} are tailored to algebras of locally compact operators, and do not directly apply to $\cpcstar{\variable; \CH}$.
In view of this, for the convenience of the reader, we prefer to include a quick proof.

\begin{proof}[Proof of \cref{thm:intro:rigidity}]
    Suppose that $\roecstar{X; \CH}\cong \roecstar{Y; \CH}$ or $\cpcstar{X; \CH}\cong\cpcstar{Y;\CH}$.
    \cref{thm: implementation,thm: uniformization} show that the isomorphism is implemented by a weakly approximately controlled unitary $U\colon \elltwo X\to\elltwo{Y}$.

    Arbitrarily fix some $0<\delta<1$. Since $\chf A$ has propagation zero for every $A\subseteq X$, weak approximability implies that for every $\varepsilon>0$ there is an $R \geq 0$ large enough so that $\Ad(U)(\chf{A})$ is $\varepsilon$-$R$-quasi-local for all \(A \subseteq X\). 
    Letting $\varepsilon = \frac 12(1-\delta^2)^{1/2}$, \cref{prop:conc ineq} implies that for every $y\in Y$ there is some $x\in X$ with $\norm{\chf{\overline{B}(y;R)}U\chf x}>\delta$. 
    Choosing one such $x \in X$ for all $y\in Y$ defines a function $g\colon Y\to X$ such that
    \begin{equation}\label{eq: g function}
    \norm{\chf{\overline{B}(y;R)}U \chf{g(y)}}>\delta \quad \text{for all } y\in Y.
    \end{equation}
    Observe that if $d_Y(y,y')\leq r$, then 
    \[
        \diam\bigparen{\overline{B}(y;R)\cup \overline{B}(y';R) }\leq 2R+r,
    \]
    so \cref{lem:approximationa are controlled} gives a uniform upper bound on $d_X(g(y),g(y'))$. That is, $g$ is a controlled map.

    Since $U^*$ implements the inverse isomorphism, it is weakly approximately controlled as well, always by \cref{thm: uniformization}. Therefore, the same argument can be used to construct a controlled map $f\colon X\to Y$ such that 
    \begin{equation}\label{eq: f function}
        \norm{\chf {f(x)} U \chf{\overline{B}(x;R)}}=\norm{\chf{\overline{B}(x;R)}U^*\chf{f(x)}}>\delta \quad \text{for all } x\in X
    \end{equation}
    (picking the largest constant if necessary, we may assume $R$ to be the same for $f$ and $g$).

    It only remains to show that $f\circ g$ and $g\circ f$ are close to the identity. \cref{eq: g function,eq: f function} show that
    \begin{align*}
        \norm{\chf{\overline{B}(f(x);R)}U \chf{g(f(x))}} &>\delta, \\
        \norm{\chf {\overline{B}(f(x);R)} U \chf{\overline{B}(x;R)}} &\geq \norm{\chf {f(x)} U \chf{\overline{B}(x;R)}} >\delta.
    \end{align*}
    By \cref{lem:approximationa are controlled} there is a uniform upper bound on $d_X\bigparen{g(f(x)), \overline{B}(x;R)}$, from which it follows that $g\circ f$ is close to the identity. A symmetric argument applies to $f\circ g$.
\end{proof}

\section{More refined results}\label{sec: refinement}
In this last section of the paper we will use two more results to obtain some more refined information regarding \cstar rigidity. This is in pursuit of some more ``functorial'' version of \cref{thm:intro:rigidity}.

We start recalling a few facts. In the following, $\CM(\roecstar{X; \CH})$ denotes the multiplier algebra of \(\roecstar{X; \CH}\), which is naturally realized as a subalgebra of $\ops{X}$.

\begin{theorem}[cf.\ \cite{braga_gelfand_duality_2022}*{Proposition 4.1}]\label{thm: multiplier}
    If \(X\) is a uniformly locally finite metric space, then \(\cpcstar{X; \CH} = \CM(\roecstar{X; \CH}).\)
\end{theorem}

\begin{theorem}[cf.\ \cite{braga_gelfand_duality_2022}*{Proposition 2.1} or \cite{roe-algs}*{Theorem 6.20}]\label{thm: intersection}
    If \(X\) is a uniformly locally finite metric space, then
    \[
        \roecstar{X; \CH} = \cpcstar{X; \CH} \cap \{\text{locally compact operators}\}.
    \]
\end{theorem}

The following is the first result of interest to us. For it, we implicitly use implementing unitaries to see both $\aut(\roecstar{X; \CH})$ and $\aut(\cpcstar{X; \CH})$ as subgroups of the unitary group of $\ops X$.
\begin{proposition} \label{prop: aut and out are the same}
    If \(X\) is a uniformly locally finite metric space, then
    \[
        \aut(\roecstar{X; \CH})=\aut(\cpcstar{X; \CH}).
    \]
    In particular, $\out(\roecstar{X; \CH})=\out(\cpcstar{X; \CH})$.
\end{proposition}
\begin{proof}
    The containment $\aut(\roecstar{X; \CH})\subseteq \aut(\cpcstar{X; \CH})$ is an immediate consequence of \cref{thm: multiplier}: an automorphism of $\roecstar X$ must extend to its multiplier algebra (see, \emph{e.g.}\ \cite{blackadar2006operator}*{II.7.3.9}), and these automorphisms must be implemented by the same unitary $U\in\ops X$.

    For the converse containment, it follows from \cref{thm: intersection} that it is enough to show that if $U$ implements an automorphism of $\cpcstar{X; \CH}$, then $\Ad(U)$ must preserve local compactness. As before, note that $U$ and $U^*$ are weakly approximately controlled by \cref{thm: uniformization}.

    Let $A\subseteq X$ be an arbitrary non-empty finite set.
    Exhausting $X$ by larger and larger finite sets, we may find some finite set $B\subseteq X$ such that $\norm{\chf B U^*\chf A}\geq 1/2$. An application of \cref{lem:approximationa are controlled} shows that for every $\varepsilon>0$ there is an $R\geq 0$ large enough such that
    \[
    \norm{\chfcX{N_R(B)}U^*\chf A} < \varepsilon,
    \]
    where $N_R(B)$ denotes the $R$-neighborhood of $B$.
    If $t\in\ops B$ is a locally compact operator, we deduce that
    \[
     \Ad(U)(t)\chf A = UtU^*\chf A
     =\lim_{R\to \infty} U \paren{t\chf{N_R(B)} } U^*\chf A
    \]
    is compact, as it is the limit of compact operators. We may analogously show that $\chf A UtU^*$ is compact as well. Since $A$ and $t$ are arbitrary, this proves that $\Ad(U)$ preserves local compactness.
\end{proof}

\begin{remark}
    The containment $\aut(\roecstar{X; \CH})\subseteq \aut(\cpcstar{X; \CH})$ is also noted in \cite{braga_gelfand_duality_2022}*{Corollary 4.3}. See \cite{braga_gelfand_duality_2022}*{Remark 4.4} for an argument not using \cref{thm: multiplier}.
\end{remark}

An operator $T\colon \elltwo{X}\to \elltwo{Y}$ is \emph{coarsely supported} on a function $f\colon X\to Y$ if and only if there is a constant $R\geq 0$ such that $\chf y T\chf x \neq 0$ only if $d_Y(f(x),y)\leq R$. If it is important to keep track of the specific constant, we will say that $T$ is \emph{$R$-supported on $f$}. 
The following refinement of \cref{thm:intro:rigidity} is the main technical result of this section.

\begin{theorem}\label{thm: refined rigidity}
    Suppose $\CH$ is infinite dimensional and $U\colon \elltwo{X}\to \elltwo{Y}$ is a unitary implementing an isomorphism of $\roecstar{\variable; \CH}$ (equivalently, $\cpcstar{\variable; \CH}$). Let $f\colon X\to Y$ be a coarse equivalence constructed as in the proof of \cref{thm:intro:rigidity}.
    Then $U$ is a norm limit of operators that are coarsely supported on $f$.
\end{theorem}

\begin{proof}
    Let $p\in\ops X$ be a finite rank projection of the form 
    \begin{equation}\label{eq:fin rk proj}
        p = p_{x_1}+\cdots +p_{x_n}
    \end{equation}
    where the $x_i\in X$ are distinct points and $p_{x_i}\leq \chf{x_i}$ is a projection onto some finite dimensional vector subspace $E_{i}\leq \CH \cong \chf{x_i}(\elltwo{X})$. Observe that such a $p$ has zero propagation.
    The following claim is the key place where we use the assumption that $\CH$ be infinite dimensional.

    \begin{claim}\label{claim: upgrade trick}
        For every $\varepsilon>0$ there is some $R\geq 0$ such that for every $p$ as in \eqref{eq:fin rk proj} there exist a unitary operator $V \in \ops X$ of propagation zero and an operator $t\in\ops X$ that is $R$-supported on $f$ satisfying $\norm{t-UVp}\leq \varepsilon$.
    \end{claim}
    \begin{proof}[Proof of \cref{claim: upgrade trick}]
        By the construction of $f$, there are $r\geq 0$ and $\delta>0$ such that $\norm{\chf{f(x)}U\chf{\overline{B}(x;r)}} > \delta$ for all \(x \in X\). We may assume that $\varepsilon<\delta$, and apply \cref{lem:approximationa are controlled} on norms of the form $\norm{\chf A U \chf x}\leq \norm{\chf{A} U\chf{\overline{B}(x;r)}}$ to deduce that there is some $R\geq 0$ large enough so that
        \begin{equation}\label{eq:bound on corner}
         \norm{\chfcX{\overline{B}(f(x); R)} U \chf x}\leq \varepsilon 
        \end{equation}
        for every $x\in X$.

        Recall that \(p = p_{x_1} + \dots + p_{x_n}\). In the following, with a slight abuse of notation, we are using \(\chf{x_i}\) to denote both a projection in \(\CB(\elltwo{X})\) and in $\CB(\ell^2(X))$.
        For each $i=1,\ldots, n$, let $C_i\coloneqq X\smallsetminus \overline{B}(f(x_i); R)$. We inductively construct unitary operators $V_i\in \CB(\CH)$ as follows. For a fixed $i$, consider the finite dimensional subspace 
        \[
        F_i\coloneqq \angles{U^*\chf{C_i}\chf{C_j}U(\chf{x_j} \otimes V_j)(E_{j})\mid 1\leq j \leq i}\leq \elltwo{X};
        \]
        and define $V_i\in \CB(\CH)$ by arbitrarily choosing a unitary operator such that $V_i(E_{i})$ is orthogonal to $\chf{x_i}(F_i)$. Namely, $V_i$ is chosen so that
        \begin{equation} \label{eq:h inf dim for this}
        p_{x_i}(\chf{x_i} \otimes V_i^*)(F_i) = \{0\}.
        \end{equation}

        Consider now the partial isometries $\chf{x_i}\otimes V_i\in\ops X$. We claim that the operators $\chf{C_i}U(\chf{x_i} \otimes V_i)p_{x_i}$ are orthogonal to one another as $i=1,\ldots,n$ varies.
        In fact, it is clear that for every $j< i$ 
        \[
            \chf{C_{j}}U(\chf{x_{j}} \otimes V_{j})p_{x_{j}}\bigparen{\chf{C_{i}}U(\chf{x_{i}} \otimes V_{i})p_{x_{i}}}^*
            = \chf{C_{j}}U(\chf{x_{j}} \otimes V_{j})p_{x_{j}}p_{x_{i}}(\chf{x_{i}} \otimes V_{i}^*)U^*\chf{C_{i}}
            =0
        \]
        (since \(p_{x_{j}} p_{x_{i}} = 0\) when \(i \neq j\)), and 
        \[
            \bigparen{\chf{C_{i}}U(\chf{x_{i}} \otimes V_{i})p_{x_{i}}}^*\chf{C_{j}}U(\chf{x_{j}} \otimes V_{j})p_{x_{j}}
            = p_{x_{i}}(\chf{x_{i}} \otimes V_{i}^*)\bigparen{U^*\chf{C_{i}}\chf{C_{j}}U(\chf{x_{j}} \otimes V_{j})p_{x_{j}}} = 0
        \]
        (from the choice of \(V_i\), see \eqref{eq:h inf dim for this}).
        
        Observe that
        \[
          V \coloneqq \sum_{i=1}^n V_i\otimes\chf{x_i} + \chfcX{\{x_1,\ldots, x_n\}},
        \]
        is a unitary operator of propagation zero. Moreover, let 
        \[
            t\coloneqq \sum_{i=1}^n \chf{\overline{B}(f(x_i);R)}U V {p_{x_i}}.
        \]
        We claim that $V$ and $t$ satisfy our requirements. It is evident that $t$ is $R$-supported on $f$, and we see that
        \[
        \norm{t-UVp}
        =\norm{\sum_{i=1}^n \chf{\overline{B}(f(x_i);R)}U V {p_{x_i}} - U V p_{x_i} }
        =\norm{\sum_{i=1}^n \chf{C_i}U (\chf{x_i} \otimes V_i){p_{x_i}}}.
        \]
        Since these operators are orthogonal by construction, this norm is equal to the maximum of $\norm{\chf{C_i}U (\chf{x_i} \otimes V_i){p_{x_i}}} $, which is at most $\varepsilon$ (cf.\ \eqref{eq:bound on corner}).
    \end{proof}

    Fix now $\varepsilon >0$.
    Choose a net $(p_\lambda)_{\lambda \in \Lambda}$ of the form \eqref{eq:fin rk proj} that converges strongly to $1\in\ops X$. Apply \cref{claim: upgrade trick} to obtain an $R_1\geq 0$, unitaries $(V_\lambda)_{\lambda \in \Lambda}$ and operators $(t_\lambda)_{\lambda \in \Lambda}$ that are $R$-supported on $f$ and such that  $\norm{t_\lambda-UV_\lambda p_\lambda}\leq \varepsilon/2$ for every $\lambda \in \Lambda$.
    Since $U$ is weakly approximately controlled and $(V_\lambda)_{\lambda \in \Lambda}$ all have propagation $0$, there is also an $R_2>0$ such that for every $\lambda \in \Lambda$ there is some $s_\lambda\in\ops X$, whose propagation is bounded by $R_2$, such that $\norm{U V_\lambda^* U^*-s_\lambda}\leq \varepsilon/2$.

    For convenience, we may also impose that each $s_\lambda$ be a contraction. It then follows from the triangle inequality that
    \[
    \norm{Up_\lambda- s_\lambda t_\lambda} 
    = \norm{\paren{UV_\lambda^* U^*}\paren{UV_\lambda p_\lambda} - s_\lambda t_\lambda}\leq \varepsilon.
    \]
    Letting $R\coloneqq R_1+R_2$, observe that the operator $s_\lambda t_\lambda$ is $R$-supported on $f$.
    Since $Up_\lambda$ converges to $U$ in the strong operator topology, this shows that $U$ is the strong limit of operators that are within distance $\varepsilon$ from operators that are $R$-supported on $f$.
    As the set
    \[
    \braces{T\in\ops X\mid \exists T'\in\ops X \text{ $R$-supported on $f$ with } \norm{T-T'}\leq \varepsilon}
    \]
    is closed in the strong operator topology (see the proof of \cite{braga_gelfand_duality_2022}*{Proposition 3.7}), it follows that $U$ itself can be $\varepsilon$-approximated with an operator that is $R$-supported on $f$. 
\end{proof}

\begin{remark}
    It is not hard to use \cref{thm: refined rigidity} to prove \cref{thm: multiplier} (under the assumption that $\CH$ be infinite dimensional).
\end{remark}

Let again $\CH$ be infinite dimensional. As is well known, with any coarse equivalence $f\colon X\to Y$ one can associate a unitary $U_f\colon \elltwo{X}\to \elltwo{Y}$ coarsely supported on $f$ (such a $U_f$ is also said to \emph{cover} $f$, see \emph{e.g.}\ \cite{willett_higher_2020}*{Proposition 4.3.4}).
This is straightforward to see if \(f\) is a \emph{bijective} coarse equivalence: then \(U_f(\delta_x \otimes h) \coloneqq \delta_{f(x)} \otimes h\) defines a well-behaved unitary covering \(f\) (this case even works if \(\CH\) is finite dimensional).
In general, one may find some coarsely dense \(X_0 \subseteq X\) such that the restriction of $f$ to $X_0$ is injective, and subsequently partition \(X\) and $Y$ into sets \(\{X(x_0)\}_{x_0 \in X_0}\) and \(\{Y(x_0)\}_{x_0 \in X_0}\) of uniformly bounded diameter such that \(x_0 \in X(x_0)\) and \(f(x_0) \in Y(x_0)\) for all \(x_0 \in X_0\).
In such setting, one may choose bijections \(g_{x_0} \colon X(x_0) \times \NN \to Y(f(x_0)) \times \NN\), since both these sets are countably infinite. For every \(x \in X(x_0)\), let \(g_{x_0,1}\left(x, n\right)\in Y(x_0)\) and \(g_{x_0,2}\left(x, n\right)\in \NN\) be the coordinates of \(g_{x_0}\left(x, n\right)\). Then the map
\[
  U_f \left(\delta_x \otimes e_n\right) \coloneqq \delta_{g_{x_0,1}\left(x, n\right)} \otimes e_{g_{x_0,2}\left(x, n\right)},
\]
where \(\{e_n\}_{n \in \NN} \subseteq \CH\) is any orthonormal basis and \(x_0 \in X_0\) is such that \(x \in X(x_0)\),
defines the desired unitary.

Note that the construction of \(U_f\) involves highly non-canonical choices (the co-bounded set \(X_0\), the partitions, the bijections\ldots). Nevertheless, it is not hard to show that different choices give rise to unitaries that only differ by composition with some unitary of controlled propagation. In turn, this shows that this procedure induces a canonically defined group homomorphism 
\[
    \begin{tikzcd}[row sep = 0, ]
        \tau \colon \coe X \arrow[r] & \out(\cpcstar X) \\
            \hspace{3ex} [f] \arrow[r,|->] & {[\Ad(U_f)]}
    \end{tikzcd},
\]
where 
\[
\coe X \coloneqq \braces{f\colon X\to X \text{ coarse equivalence}}/_\text{closeness}
\]
(this is a group under composition).
With a little more work, one can even show that $\tau$ is, in fact, injective (cf.\ \cite{braga_gelfand_duality_2022}*{Section 2.2} or \cite{roe-algs}*{Theorem 7.18}). Using \cref{thm: refined rigidity} we can now show that it is even an isomorphism, proving \cref{thm:intro:iso} in the introduction.

\begin{proof}[Proof of \cref{thm:intro:iso}]
    By the discussion above, all it remains to do is to check that
    \[
        \tau \colon \coe{X} \to \out\left(\roecstar{X}\right) = \out\left(\cpcstar{X}\right)
    \]
    is surjective.
    Let $U\in\ops X$ be a unitary implementing an automorphism (of \(\roecstar{X}\) or \(\cpcstar{X}\), see \cref{prop: aut and out are the same}). By \cref{thm:intro:rigidity}, we can construct an associated coarse equivalence $f\colon X\to Y$, and by \cref{thm: refined rigidity} there exists a sequence of operators $T_n$ that are coarsely supported on $f$ and converge to $U$ in norm.

    Fix some unitary $W\in \ops X$ coarsely supported on $f$. Observe that if $\chf {x'} T_nW^*\chf x\neq 0$ then there must be some $\bar x\in X$ such that  $\chf {x'} T_n\chf{\bar x}W^*\chf x\neq 0$.
    Since $W$ and $T_n$ are both coarsely supported on $f$, it follows that both $x'$ and $x$ are at uniformly bounded distance from $f(\bar x)$. In particular, $d_X(x',x)$ is uniformly bounded as \(x, x' \in X\) vary with \(\chf{x'} T_n W^* \chf{x} \neq 0\). This means precisely that for every fixed $n\in \NN$ the operator $T_nW^*$ has bounded propagation.
    Now we are done, because $U = (UW^*)W$ and 
    \[
        UW^* = \lim_{n\to \infty} T_nW^*
    \]
    is therefore a unitary in $\cpcstar X$, so $[\Ad(U)]=[\Ad(W)]$ (in \(\out(\roecstar{X})\)). Moreover, \([\Ad(W)]\) is in the image of $\tau$ by construction.
\end{proof}

\bibliography{BibRigidity}

\end{document}